\documentclass[11pt,letterpaper]{amsart}
\usepackage[utf8]{inputenc}
\usepackage{stmaryrd}
\usepackage{bbm}
\usepackage[dvipsnames]{xcolor}
\usepackage{mathrsfs}
\usepackage{enumerate}
\usepackage{centernot}
\usepackage{latexsym,amsxtra}
\usepackage{dsfont}
\usepackage{slashed}
\usepackage[all]{xy}
\usepackage{amscd,graphics}
\usepackage{amsmath,amsfonts,amsthm,amssymb}
\usepackage{latexsym,amsmath}
\usepackage{graphicx,psfrag}
\usepackage{mathabx}
\usepackage{hyperref}
\usepackage{fancyhdr}

\usepackage{cleveref}
\usepackage{comment} 
\usepackage{braket}
\usepackage{caption}
\usepackage{enumitem}

\theoremstyle{plain}
\newtheorem{thm}{Theorem}[section]
\crefname{thm}{Theorem}{Theorems}
\Crefname{thm}{Theorem}{Theorems}
\newtheorem{pro}[thm]{Proposition}
\crefname{pro}{Proposition}{Propositions}
\Crefname{pro}{Proposition}{Propositions}
\newtheorem{lem}[thm]{Lemma}
\crefname{lem}{Lemma}{Lemmas}
\Crefname{lem}{Lemma}{Lemmas}
\newtheorem{cor}[thm]{Corollary}
\crefname{cor}{Corollary}{Corollaries}
\Crefname{cor}{Corollary}{Corollaries}

\crefname{conj}{Conjecture}{Conjectures}
\Crefname{conj}{Conjecture}{Conjectures}

\crefname{cons}{Construction}{Constructions}
\Crefname{cons}{Construction}{Constructions}

\crefname{claim}{Claim}{Claims}
\Crefname{claim}{Claim}{Claims}

\crefname{property}{Property}{Properties}
\Crefname{property}{Property}{Properties}

\crefname{problem}{Problem}{Problems}
\Crefname{problem}{Problem}{Problems}

\theoremstyle{definition}

\crefname{defi}{Definition}{Definitions}
\Crefname{defi}{Definition}{Definitions}

\crefname{nota}{Notation}{Notations}
\Crefname{nota}{Notation}{Notations}

\crefname{convention}{Convention}{Conventions}
\Crefname{convention}{Convention}{Conventions}

\crefname{cond}{Condition}{Conditions}
\Crefname{cond}{Condition}{Conditions}

\crefname{assum}{Assumption}{Assumptions}
\Crefname{assum}{Assumption}{Assumptions}

\theoremstyle{remark}
\newtheorem{rmk}[thm]{Remark}
\crefname{rmk}{Remark}{Remarks}
\Crefname{rmk}{Remark}{Remarks}

\crefname{ex}{Example}{Examples}
\Crefname{ex}{Example}{Examples}

\crefname{ques}{Question}{Questions}
\Crefname{ques}{Question}{Questions}

\crefname{section}{Section}{Sections}
\Crefname{section}{Section}{Sections}
\crefname{subsection}{Subsection}{Subsections}
\Crefname{subsection}{Subsection}{Subsections}
\crefname{figure}{Figure}{Figures}
\Crefname{figure}{Figure}{Figures}

\newcommand{\Diff}{\textnormal{Diff}}

\newcommand{\Gr}{\textnormal{Gr}}

\newcommand{\inv}{\textnormal{inv}}

\newcommand{\SW}{\mathrm{SW}}

\newcommand{\id}{\textnormal{id}}
\newcommand{\diag}{\textnormal{diag}}
\newcommand{\Z}{\mathbb{Z}}

\newcommand{\R}{\mathbb{R}}

\newcommand{\CP}{\mathbb{CP}}

\newcommand{\fraks}{\mathfrak{s}}

\DeclareMathOperator{\Aut}{Aut}

\newcommand{\Dsla}{\slashed{D}}

\title[Exotic diffeomorphisms of irreducible 4-manifolds]{Irreducible 4-manifolds can admit exotic diffeomorphisms}
\author{David Baraglia}
\author{Hokuto Konno}

\address{School of Computer and Mathematical Sciences, The University of Adelaide, Adelaide SA 5005, Australia}
\email{david.baraglia@adelaide.edu.au}

\address{Graduate School of Mathematical Sciences, the University of Tokyo, 3-8-1 Komaba, Meguro, Tokyo 153-8914, Japan, and, RIKEN iTHEMS, Wako, Saitama 351-0198, Japan}
\email{konno@ms.u-tokyo.ac.jp}

\begin{document}

\maketitle

\begin{abstract}
We prove that a variety of examples of minimal complex surfaces admit exotic diffeomorphisms, providing the first known instances of exotic diffeomorphisms of irreducible 4-manifolds. We also give sufficient conditions for the boundary Dehn twist on a spin $4$-manifold with $S^3$ boundary to be non-trivial in the relative mapping class group. This gives many new examples of non-trivial boundary Dehn twists.
\end{abstract}

\section{Introduction}

\subsection{Main results}

Given a smooth manifold $X$, a diffeomorphism $f : X \to X$ is said to be an {\it exotic diffeomorphism} if $f$ is topologically isotopic to the identity but not smoothly.
After Ruberman's pioneering work \cite{Rub98}, extensive research has been conducted to the detection of exotic diffeomorphisms of 4-manifolds, with particularly active developments in recent years
\cite{BK20gluing,KM-dehn,JLinstabilization23,iida2024diffeomorphisms4manifoldsboundaryexotic,konno-mallick-taniguchi,Baraglia23mapping,KLMME,KangParkTaniguchi,miyazawaDehn,qiu2024surgeryformulasseibergwitteninvariants,KLMME2}.

Despite significant interest from experts, the fundamental question of whether an irreducible 4-manifold can admit an exotic diffeomorphism remains a major unresolved problem in the field.
Here a 4-manifold is said to be {\it irreducible} if it cannot decompose as a connected sum unless one summand is a homotopy 4-sphere.
Irreducible 4-manifolds naturally emerge as an important class, including minimal symplectic 4-manifolds \cite[Theorem 5.4]{Kotschick97}.
Many of them are known to admit exotic structures, making irreducible 4-manifolds a central focus of interest in 4-dimensional topology. 
The purpose of this paper is to resolve the above open problem affirmatively:

\begin{thm}
There exist irreducible closed smooth 4-manifolds that admit exotic diffeomorphisms.    
\end{thm}

More concretely, we show that a variety of examples of simply-connected minimal complex surfaces admit exotic diffeomorphisms.
Using the standard notation for logarithmic transformations of elliptic surfaces $E(n)$
(see \cref{subsection Elliptic} for details), we can summarize the result as follows:

\begin{thm}
\label{main thm intro}
Let $X$ be one of:
\begin{itemize}
    \item $E(4m)_{i,j}$, $m \ge 1$, $j \ge i \ge 1$, $i,j$ are odd, coprime and $(i,j) \notin S_1$, where
    \[
    S_1 = \{(1,1), (1,3) , (1,5) , (1,7) , (1,9) , (3,5)\}.
    \]
     \item A complete intersection $X$ such that $c_1(X)$ and $\sigma(X)$ are divisible by 32.
\end{itemize}
Then $X$ admits an exotic diffeomorphism.
\end{thm}

The above list of complex surfaces includes both elliptic surfaces and surfaces of general type, demonstrating that exotic diffeomorphisms of complex surfaces exist with considerable generality. All of our examples are spin manifolds.

\cref{main thm intro} follows from the following result on the structure of mapping class groups.
Given a closed oriented 4-manifold $X$, let $Q_X$ denote the intersection form and let $\Diff^+(X)$ denote the group of orientation-preserving diffeomorphisms.
Let $\Gamma(X)$ denote the image of the natural homomorphism $\Diff^+(X) \to \Aut(Q_X)$.

\begin{thm}
\label{main thm intro2}
Let $X$ be as in \cref{main thm intro}.
Then the surjective homomorphism 
\[
\pi_0(\Diff^+(X)) \to \Gamma(X)
\]
does not split.
\end{thm}

\cref{main thm intro2} makes an interesting contrast to $K3=E(2)$, as $\pi_0(\Diff(K3)) \to \Gamma(K3)$ {\it does} split \cite[Theorem 1.1]{BKnielsen}.
Thus, an analog of \cref{main thm intro2} does not hold for general complex surfaces, not even within elliptic surfaces.

\cref{main thm intro2} is proven by showing that there is a homomorphism $\Z^2 \to \Gamma(X)$ that does not lift along the natural homomorphism $\pi_0(\Diff^+(X)) \to \Gamma(X)$, and the exotic diffeomorphism given in \cref{main thm intro} is the form of a commutator $[f_1,f_2]$, where $f_i \in \Diff^+(X)$.

A secondary aim of this paper concerns the non-triviality of boundary Dehn twists. A {\it (4-dimensional) Dehn twist} is a diffeomorphism $\tau : [0,1] \times S^3 \to [0,1] \times S^3$ of the form $\tau(t,y) = (t , \alpha_t(y))$, where $\alpha : [0,1] \to SO(4)$ is a smooth loop based at the identity which represents the non-trivial class in $\pi_1(SO(4))$. More generally, if $X$ is a $4$-manifold with an embedded copy of $[0,1] \times S^3$, then we can extend $\tau$ to a diffeomorphism of $X$ by taking it to be the identity outside of $[0,1] \times S^3$. In particular, if $X$ is a $4$-manifold with $S^3$-boundary, then we can consider the Dehn twist in a collar neighborhood of $\partial X$. We call this the {\it boundary Dehn twist} of $X$.

If $X$ is a closed, simply-connected $4$-manifold and $X_0$ is obtained from $X$ by removing an open ball, then $X_0$ has $S^3$ boundary and we can consider the boundary Dehn twist. Let $\Diff(X_0 , \partial)$ denote the group of diffeomorphisms of $X_0$ which are the identity in a neighbourhood of the boundary. Denote by $t_X$ the class of the boundary Dehn twist in $\pi_0(\Diff(X_0 , \partial))$. It is known that the map
\[
\pi_0(\Diff(X_0 , \partial) ) \to \pi_0(\Diff^{+}(X_0))
\]
is surjective and the kernel is generated by $t_X$ \cite[Corollary 2.5]{Gian08}. Hence the kernel is either trivial or isomorphic to $\mathbb{Z}_2$ according to whether or not $t_X$ is trivial. It is known that $t_X$ is trivial if $X$ is non-spin \cite[Corollary A.5]{OP22}. It is also known that $t_X$ is {\it topologically} isotopic to the identity relative boundary for any simply-connected 4-manifold \cite[Theorem E]{OP22}. Thus, when $t_X$ is smoothly non-trivial it provides an example of a {\it relative exotic diffeomorphism}. In the spin case little is known about the triviality or non-triviality of $t_X$ (cf. \cite[Question 1.1]{OP22}). It is easily seen that $t_X$ is trivial for connected sums of $S^2 \times S^2$ by considering circle actions. For $X = K3$ it follows from {\cite[Corollary 1.3]{BaragliaKonnofamiliesBF}} that $t_X$ is non-trivial (see \cite[Proposition 1.2]{KM-dehn}). More generally if $X$ is homeomorphic to $K3$, the same argument shows non-triviality of $t_X$. Beyond this the only other case where $t_X$ was known to be non-trivial is if $X$ is homeomorphic to $K3 \# (S^2 \times S^2)$ \cite[Theorem 7.2]{Baraglia23mapping}, \cite{JLinstabilization23}, \cite
[Theorem  5.14]{konno-mallick-taniguchi}. Our next result provides many new examples of irreducible spin $4$-manifolds for which the boundary Dehn twist is non-trivial:

\begin{thm}\label{thm:dehntw-intro}
Let $X$ be a compact, simply-connected, smooth spin $4$-manifold. Suppose that $\mathfrak{s}$ is a spin$^c$ structure on $X$ such that $c_1(\mathfrak{s})$ is divisible by $32$ and $\SW(X , \mathfrak{s})$ is odd. Assume also that $\sigma(X) = 16$ {\rm mod 32} and $b_+(X) = 3$ {\rm mod 4}. Let $X_0$ be $X$ with an open ball removed. Then the Dehn twist on the boundary of $X_0$ is non-trivial in $\pi_0(\Diff(X_0 , \partial))$.
\end{thm}

Examples of $4$-manifolds satisfying these assumptions are easily obtained by considering elliptic surfaces or complete intersections:

\begin{thm}
\label{main thm intro 3}
Let $X$ be one of:
\begin{itemize}
    \item $E(4m-2)_{i,j}$, $m \ge 1$, $j \ge i \ge 1$, $i,j$ are odd, coprime and $(i,j) \notin S_2$, where
    \[
    S_2 = \{ (1,j) \}_{j \le 15} \cup \{ (3,5) , (3,7) , (7,9), (5,11), (3,13)\}.
    \]
     \item A complete intersection $X$ such that $c_1(X)$ is divisible by 32 and $\sigma(X) = 16$ {\rm mod 32}.
\end{itemize}
Let $X_0$ be $X$ with an open ball removed. Then the Dehn twist on the boundary of $X_0$ is non-trivial in $\pi_0(\Diff(X_0 , \partial))$.
\end{thm}

\subsection{Methods}

Let us recall why it is challenging to detect exotic diffeomorphisms of irreducible 4-manifolds.
The principal method for detecting an exotic diffeomorphism of a 4-manifold is by computing the 1-parameter families Seiberg--Witten invariants.
However, irreducibility is typically detected by the Seiberg--Witten invariant, which, for formal dimensional reasons, implies the vanishing of the 1-parameter families Seiberg--Witten invariants.

The method we use to detect exotic diffeomorphisms in this paper, instead, makes use of a constraint on smooth families of 4-manifolds established in a previous work by the authors \cite{BaragliaKonnofamiliesBF}, which is based on families Seiberg--Witten theory.
The main results are derived from this constraint, in combination with the families index theorem and classical realization results of lattice automorphisms as diffeomorphisms of complex surfaces.

\subsection{Acknowledgments.}
We would like to thank the anonymous referee for their valuable comments.
D. Baraglia was financially supported by an Australian Research Council Future Fellowship, FT230100092.
H. Konno was partially supported by JSPS KAKENHI Grant Numbers 21K13785 and 25K00908.

\section{Obstruction to smooth isotopy}
\label{section Obstruction to smooth isotopy}

In this section, we give an obstruction to smooth isotopy.

\subsection{Vanishing of $c_1$}

Let $X$ be an oriented smooth manifold and $\fraks$ be a spin$^c$ structure on $X$.
Given a smooth fiber bundle $E \to B$ with fiber $X$,
we say that $E$ is a smooth family of spin$^c$ manifolds with fiber $(X,\fraks)$ if $E$ is equipped with a spin$^c$ structure $\fraks_E$ on the vertical tangent bundle $TE/B$ that restricts to $\fraks$ on the fibers.
We denote by $\slashed{D}_E \in K^0(B)$ the families index of the family of spin$^c$ Dirac operators determined by $(E , \fraks_E)$.

\begin{pro}
\label{prop: divisibility}
Let $X$ be a closed oriented, simply-connected smooth 4-manifold and $\fraks$ be a spin$^c$ structure on $X$.
Let $B$ be a closed orientable surface and let $(E , \fraks_E)$ be a smooth family of spin$^c$ 4-manifolds over $B$ with fiber $(X,\fraks)$.
If both $\sigma(X)$ and $c_1(\fraks)$ are divisible by 32, then we have $c_1(\Dsla_E)=0$ {\rm mod 2}.
\end{pro}

\begin{proof}
Set $c = c_1(\fraks_E),\ p_1=p_1(TE/B) \in H^\ast(E;\Z)$. Let $\int_{E/B} : H^\ast(E) \to H^\ast(B)$ denote the fiber integration. Since $H^2(B ; \mathbb{Z}) \cong \mathbb{Z}$ has no torsion, it suffices to compute the image of $c_1( \Dsla_E )$ in rational cohomology. For this we can use the Atiyah--Singer index theorem for families, which gives:
\[
Ch(\Dsla_E) = rk(\Dsla_E) + c_1(\Dsla_E) = \int_{E/B} e^{c/2} \widehat{A}(TE/B) \in H^*(B ; \mathbb{Q}).
\]
Since $e^{c/2} = 1 + c/2 + c^2/8 + c^3/48$ and $\widehat{A}(TE/B) = 1 - p_1/24$, we find
\[
c_1(\Dsla_E)
= \frac{1}{48}\int_{E/B}(c^3 - cp_1).
\]
Writing $3c_1(\Dsla_E) = \frac{1}{16} \int_{E/B} (c^3 - cp_1)$, we see that if $\int_{E/B}c^3$ and $\int_{E/B}cp_1$ are divisible by 32, then $c_1(\Dsla_E) = 0$ mod 2.

For the remainder of the proof, we use $H^\ast(-)$ to denote cohomology with coefficients in $\Z$. The Serre spectral sequence for $E \to B$ has no differentials for degree reasons, thus we get an exact sequence
\[
0 \to H^2(B) \to H^2(E) \xrightarrow{r} H^2(X)^{\inv} \to 0,
\]
where $H^2(X)^{\inv}$ denotes the invariant part of $H^2(X)$ under the monodromy action of $E \to B$, and $r : H^2(E) \to H^2(X)^{\inv}$ is the restriction map.
As $H^2(X)^{\inv}$ is a free abelian group, we have a (non-canonical) splitting
\[
H^2(E) \cong H^2(B) \oplus H^2(X)^{\inv}.
\]
The image of $c$ under $r : H^2(E) \to H^2(X)^\inv$ is $c_1(\fraks)$ which is assumed to be divisible by 32 in $H^2(X)$.
Since $H^2(X)$ is torsion free, $c_1(\fraks)$ is divisible also in $H^2(X)^\inv$.
Thus we can write $c = \pi^{\ast}(b)+32c_0$ for some $b \in H^2(B)$, $c_0 \in H^2(E)$.
Here $c_0$ is the image of $c_1(\fraks)/32$ under a choice of splitting $H^2(X)^\inv \to H^2(E)$, and $\pi : E \to B$ is the projection.
Now, noting $b^2=0$, we have
\begin{align*}  
\int_{E/B}c^3
= \int_{E/B}(\pi^{\ast}(b)+32c_0)^3
= \int_{E/B}\left(3\pi^{\ast}(b)(32c_0)^2 + (32c_0)^3\right) = 0 \; {\rm mod \; 32}
\end{align*}
and
\begin{align*}  
\int_{E/B}cp_1
= \int_{E/B}(\pi^{\ast}(b)+32c_0)p_1
= b \int_{E/B}p_1
= 3\sigma(X)b = 0 \; {\rm mod \; 32},
\end{align*}
where we have used that $\sigma(X)$ is divisible by 32 at the last step.
This completes the proof.
\end{proof}

\subsection{Main obstruction}

Given $p,q\geq0$, let $H \to B$ be a vector bundle with structure group $O(p,q)$.
Let $\Gr_+(p,q)$ denote the Grassmannian of maximal-dimensional positive-definite subspaces of $\R^{p+q}$ equipped with the standard indefinite form of signature $(p,q)$.
As $\Gr_+(p,q)$ can be identified with the quotient by the maximal compact subgroup $O(p,q)/(O(p) \times O(q))$, it follows that $\Gr_+(p,q)$ is contractible.
Thus the fiber bundle $\Gr_+(H) \to B$ with fiber $\Gr_+(p,q)$ associated with $H \to B$ has a section unique up to homotopy.
A choice of section defines a subbundle $H^+ \to B$ of $H$, and any other choice results in an isomorphic vector bundle. Thus, we often omit the choice in our notation.
We call $H^+$ the maximal positive-definite subbundle of $H \to B$.

If $H \to B$ arises as the vector bundle with fiber $H^2(X;\R)$ associated with an oriented fiber bundle $E \to B$ with fiber oriented closed 4-manifold $X$, we denote $H^+$ by $H^+(E)$. The key ingredient in this paper is:

\begin{thm}[{\cite[Corollary 1.3]{BaragliaKonnofamiliesBF}}]
\label{thm: BK input}
Let $X$ be a closed oriented smooth 4-manifold with $b_+(X) = 3$ {\rm mod 4} and $b_1(X)=0$.
Let $\fraks$ be a spin$^c$ structure on $X$ and let $E \to B$ be a smooth family of spin$^c$ 4-manifolds with fiber $(X,\fraks)$ over a compact base space $B$.
If the Seiberg--Witten invariant $\SW(X,\fraks)$ is odd, then we have
\[
c_1(\Dsla_E)=w_2(H^+(E)) \; {\rm mod \; 2}.
\]
\end{thm}

The most general result to provide exotic diffeomorphism presented in this paper is the following:

\begin{thm}
\label{general thm}
Let $X$ be a closed oriented smooth simply-connected 4-manifold with $b_+(X) = 3$ {\rm mod 4}.
Let $\fraks$ be a spin$^c$ structure on $X$ such that $\SW(X,\fraks)$ is odd.
Assume that $\sigma(X)$ and $c_1(\fraks)$ are divisible by 32 (note that the divisibility of $c_1(\fraks)$ forces $X$ to be spin).

Suppose further that there exist orientation-preserving diffeomorphisms $f_1, f_2 : X \to X$ that satisfy the following conditions:
\begin{itemize}
\item[(i)] Each $f_i$ preserves $\fraks$.
\item[(ii)] The induced actions $f_i^\ast : H^2(X;\Z) \to H^2(X;\Z)$ commute with each other.
Note that this condition induces a vector bundle $H \to T^2$ with fiber $H^2(X;\R)$ with monodromy $f_1^\ast, f_2^\ast$.
\item[(iii)] Let $H^+ \to T^2$ be a maximal positive-definite subbundle of $H$.
Then $w_2(H^+)\neq0$.
\end{itemize}
Then $[f_1,f_2] \in \Diff(X)$ is an exotic diffeomorphism.
\end{thm}

\begin{proof}
Set $f = [f_1,f_2]$.
The assertion that $f$ is topologically isotopic to the identity follows from $f_\ast=\id$ (by (ii)) and Quinn's result \cite{Q86} (see also a recent correction \cite{gabai2023pseudoisotopiessimplyconnected4manifolds}).
Thus, it suffices to prove that $f$ is not smoothly isotopic to the identity.

Suppose on the contrary that $f$ is smoothly isotopic to the identity.
Then we can get an oriented smooth fiber bundle $E \to T^2$ with fiber $X$ as follows.
Consider a standard cellular decomposition of $T^2$, $T^2 = e^0 \cup e^1_1 \cup e^1_2 \cup e^2$.
For each $i$, consider the mapping torus of $f_i$ over $e^0 \cup e^1_i = S^1$.
By wedging them, we get a smooth family of $X$ over the 1-skeleton $(T^2)^{(1)}$.
Since $f$ is supposed to be smoothly isotopic to the identity, by choosing such an isotopy, we can extend this family to $T^2$ as a smooth fiber bundle.
By $f_i^\ast\fraks\cong \fraks$, $\dim T^2<3$, and $b_1(X)=0$, it follows from \cite[Proposition 2.1]{BaragliaObstructions} that $E$ admits a families spin$^c$ structure that restricts to $\fraks$ on the fiber.
Applying \cref{thm: BK input} to the smooth family $E$ of spin$^c$ 4-manifolds, we obtain that $c_1(\Dsla_E)=w_2(H^+)$ mod 2. Since $w_2(H^+)$ is assumed to be non-trivial, this contradicts \cref{prop: divisibility}.
\end{proof}

\begin{cor}
\label{cor nonsplitting}
Let $X$ be as in \cref{general thm}.
Then the natural homomorphism
\[
\pi_0(\Diff^+(X)) \to \Aut(Q_X)
\]
does not split over its image.
\end{cor}

\begin{proof}
It follows from \cref{general thm} that the homomorphism $\Z^2 \to \Aut(Q_X)$ given by sending generators of $\Z^2$ to $(f_i)_\ast$ does not lift along $\pi_0(\Diff^+(X)) \to \Aut(Q_X)$, which implies the assertion.
\end{proof}

\begin{rmk}
The simple-connectivity is used in \cref{general thm} only to use Quinn's result \cite{Q86} to get a topological isotopy.
To obstruct smooth isotopy, one can relax the condition to $b_1(X)=0$, as
the results from \cite{BaragliaKonnofamiliesBF} and \cite{BaragliaObstructions} work for $b_1(X)=0$. 
\end{rmk}

\begin{rmk}
If $X$ has simple type then the assumption that $b_+(X) = 3$ mod $4$ in \cref{general thm} is superfluous. Indeed if $X$ has simple type then the dimension of the Seiberg--Witten moduli space for $(X , \fraks)$ is zero, hence $0 = (c_1(\fraks)^2 - \sigma(X))/8 - 1 - b_+(X)$. But $c_1(\fraks)$ and $\sigma(X)$ are divisible by $32$, so $b_+(X)+1 = 0$ mod $8$, so that $b_+(X) = 3$ mod 4. In fact we have the stronger condition $b_+(X) = 7$ mod 8.
\end{rmk}

\section{Finding diffeomorphisms}

To use our obstruction (\cref{general thm}) in practice, we shall find favorable diffeomorphisms $f_1, f_2 : X\to X$ for some class of complex surfaces $X$.
Let $H$ denote the lattice that is given by the intersection form of $S^2\times S^2$.

\begin{lem}
\label{lem: split lattice}
Let $L$ be a unimodular even indefinite form and let $c \in L$.
Then there exist a unimodular even lattice $L_0$ and an isomorphism
\[
\Phi : L \to H \oplus L_0
\]
of lattices such that $\Phi(c)$ lies in $H$.
\end{lem}

\begin{proof}
The result is immediate if $c=0$, so we assume $c \neq 0$. Let $d>0$ denote the divisibility of $c$ and write $c=dc_0$, where $c_0 \in L$ is a primitive element.
Since $c_0$ is primitive and $L$ is even, $c_0$ is not characteristic. Let $c_0^2 = 2m$.

By the classification of indefinite unimodular forms, $L$ is isomorphic to $H \oplus L_0$ for some unimodular even lattice $L_{0}$.
Using the standard basis $x,y$ of $H$ with $x^2=y^2=0, x\cdot y=1$, the element $mx+y \in H$ is primitive and has self-intersection $2m$.
By a result by Wall~\cite[page 337]{Wall62unimodular}, an element of $L$ is classified by its divisibility, self-intersection, and type (i.e. characteristic or not) up to isomorphisms of $L$.
Thus there is an isomorphism $\Phi : L \to H \oplus L_0$ such that $\Phi(c_0) = mx+y$.
This $\Phi$ is the desired isomorphism.
\end{proof}

Now we consider concrete classes of complex surfaces.
The following theorem is due to L\"onne~\cite{Lonnediffeomorphism98} for elliptic surfaces and Ebeling--Okonek~\cite{EbelingOkonekdiffeomorphism91} for complete intersections, respectively.

\begin{thm}[{\cite[Main Theorem]{Lonnediffeomorphism98}, \cite[Theorem~1]{EbelingOkonekdiffeomorphism91}}]
\label{thm: realization}
Let $X$ be a simply-connected complex surface with $b_+(X)\geq3$.
Suppose that $X$ is either a minimal elliptic surface or a complete intersection.
Let $\fraks_0$ be the canonical spin$^c$ structure.
If $\varphi \in \Aut(Q_X)$ satisfies that $\varphi^\ast \fraks_0 \cong \fraks_0$ and that $\varphi$ preserves orientation of $H^+(X)$, then $\varphi$ lies in the image of $\Diff^+(X) \to \Aut(Q_X)$.
\end{thm}

\begin{pro}
\label{pro found diffeo}
Let $X$ be as in \cref{thm: realization}. Suppose further that $X$ is spin and $b_+(X) > 3$. Then for the canonical spin$^c$ structure $\fraks=\fraks_0$ on $X$, there exist orientation-preserving diffeomorphisms $f_1, f_2 : X \to X$ that satisfy the conditions (i)-(iii) in \cref{general thm}.
\end{pro}

\begin{proof}

Using \cref{lem: split lattice}, we may take an isomorphism $Q_X \cong H \oplus L_0$ so that $c_1(\fraks)$ is mapped into $H$. Since $L_0$ is an even form that contains at least three copies of $H$, we can write $L_0=3H\oplus L_1$, where $L_1$ is a unimodular form.
Define $\varphi_i \in \Aut(L_0)$ using this expression of $L_0$ by 
\begin{align*}
\varphi_1 =\diag(-1,-1,1) \oplus \id_{L_1},\\
\varphi_2 =\diag(1,-1,-1) \oplus \id_{L_1},
\end{align*}
and extend them to elements of $\Aut(Q_X)$ by the identity.
Then it is clear that $\varphi_i$ preserve $c_1(\fraks)$, $[\varphi_1, \varphi_2]=1$, and that $\varphi_i$ preserve orientation of $H^+(X)$.
In particular, it follows from \cref{thm: realization} that there exist orientation-preserving diffeomorphisms $f_i : X \to X$ with $f_i^\ast = \varphi_i$.

Thus it suffices to check that $w_2(H^+)\neq0$.
Let $x_1, x_2 \in H^1(T^2;\Z/2)$ be the standard basis with $x_i^2=0$, $x_1x_2 \neq 0$.
The vector bundle $H^+$ is isomorphic to 
\begin{align}
\label{eq: Hplus family decomposition}
\gamma_1 \oplus \gamma_{12} \oplus \gamma_2 \oplus \underline{\R}^{b_+(X)-3},
\end{align}
where $\gamma_i, \gamma_{12}$ are real line bundles over $T^2$ with $w_1(\gamma_i)=x_i$ and $w_1(\gamma_{12})=x_1+x_2$.
Indeed, $H^+$ is determined by the
bundle $H^2 \to T^2$ with fiber $H^2(X;\R)$ associated to the family of 4-manifolds obtained from $f_i$. Since the structure group of $H^2$ is $\Aut(Q_X)$, $H^2$ is a local system and determined by its monodromy.
Therefore, we get a decomposition of $H^2$ which induces \eqref{eq: Hplus family decomposition}.
From \eqref{eq: Hplus family decomposition}, we have
\[
w(H^+) = (1+x_1)(1+x_1+x_2)(1+x_2) = 1+x_1x_2,
\]
hence $w_2(H^+)=x_1x_2\neq0$. This completes the proof.
\end{proof}

\section{Boundary Dehn twists}

In this section we prove \cref{thm:dehntw-intro}.

\subsection{Dehn twists as a commutators}

Let $\mathbb{Z}_2 \times \mathbb{Z}_2 = \langle \sigma_1 | \sigma_1^2 \rangle \times \langle \sigma_2 | \sigma_2^2 \rangle$ act on $\mathbb{R}^4$ according to $\sigma_1 = {\rm diag}(1,-1,1,-1)$, $\sigma_2 = {\rm diag}(-1,-1,1,1)$. Let $D$ denote the closed unit ball in $\mathbb{R}^4$, hence $\sigma_1, \sigma_2$ define commuting diffeomorphisms of $D$. We will construct diffeomorphisms $\sigma'_1, \sigma'_2$ which agree with $\sigma_1, \sigma_2$ for $|x| \le 1/3$ and equal the identity for $|x| \ge 2/3$. Furthermore, we will show that the commutator $[\sigma'_1, \sigma'_2]$ is smoothly isotopic to the boundary Dehn twist on $D$.

Let $t : [0,1] \to [0,1]$ be a smooth increasing function which is zero on $[0,1/3]$ and is $1$ on $[2/3,1]$. Let $h_1 : [0,1] \to SO(4)$ denote a smooth path from $\sigma_1$ to the identity. Let $e_1,e_2,e_3,e_4$ denote the standard basis of $\mathbb{R}^4$. Then $\sigma_2$ is a rotation by $\pi$ in the $e_1,e_2$-plane. Let $h_2 : [0,1] \to SO(4)$ be defined by taking $h_2(t)$ to be rotation by $(1-t)\pi$ in the $e_1,e_2$-plane. So $h_2$ is a smooth path from $\sigma_2$ to the identity. We note for later that $\sigma_1  h_2 \sigma_1^{-1} = h_2^{-1}$. Define $\sigma'_i : D \to D$ to be $\sigma'_i(x) = (h_i( t(|x|) ))(x)$. Then $\sigma'_i$ agrees with $\sigma_i$ for $|x| \le 1/3$ and equals the identity for $|x| \ge 2/3$.

We claim that $\tau = [\sigma'_1 , \sigma'_2 ]$ is smoothly isotopic to the boundary Dehn twist on $D$. From the definitions of $\sigma'_1, \sigma'_2$, we have that $\tau(x) = (p( t(|x|) ))(x)$ where $p : [0,1] \to SO(4)$ is the path given by $p = [ h_1 , h_2 ] = (h_1 h_2 h_1^{-1})h_2^{-1}$. Observe that $p$ is a closed loop in $SO(4)$ based at the identity. To show that $\tau$ is isotopic to the boundary Dehn twist it suffices to show that $p$ represents the non-trivial element of $\pi_1(SO(4)) \cong \mathbb{Z}_2$. Write $p = q h_2^{-1}$ where $q = h_1 h_2 h_1^{-1}$. For $s \in [0,1]$, let $(h_1)_s : [0,1] \to SO(4)$ be defined by $(h_1)_s(t) = h_1(t)$ for $t \le s$ and $h_s(t) = h_1(s)$ for $t \ge s$. Then $(h_1)_0 = \sigma_1$, $(h_1)_1 = h_1$. Now consider $q_s : [0,1] \to SO(4)$ given by $q_s = (h_1)_s h_2 (h_1)_s^{-1}$. This is a homotopy of paths with $q_0 = \sigma_1 h_2 \sigma_1^{-1} = h_2^{-1}$ and $q_1 = h_1 h_2 h_1^{-1} = q$. Consider endpoints of the homotopy. For $t=0$ we get $q_s(0) = (h_1)_s(0) h_2(0) (h_1)_s(0)^{-1} = \sigma_1 \sigma_2 \sigma_1^{-1} = \sigma_2^{-1}$ and for $t=1$ we get $q_s(1) = (h_1)_s(1) h_2(1) (h_1)_s(1)^{-1} = h_1(s) \id h_1(s)^{-1} = \id$. Hence $q_s$ is a homotopy relative endpoints from $h_2^{-1}$ to $q$. Therefore $q_s h_2^{-1}$ is a homotopy relative endpoints from $q_0 h_2^{-1} = h_2^{-2}$ to $q_1 h_2 = q h_2 = p$. Hence $p$ is homotopic relative endpoints to $h_2^{-2}$. But from the definition of $h_2$ it is clear that $h_2^{-2}$ represents the non-trivial element of $\pi_1(SO(4))$, proving our claim that $\tau$ is isotopic to the boundary Dehn twist.

\subsection{Non-triviality}

Let $X$ be a compact, simply-connected, smooth spin $4$-manifold. Let $x$ be a point in $X$ and choose a closed ball in $X$ centered at $x$ which we identify with the unit ball $D$ in $\mathbb{R}^4$. Let $X_0$ denote $X$ with the interior of $D$ removed. Recall that we have constructed diffeomorphisms $\sigma'_1, \sigma'_2 : D \to D$ which are trivial on a neighbourhood of $\partial D$, fix $0 \in D$, act in a neighbourhood of $0$ as $\sigma_1, \sigma_2$ and such that $[\sigma'_1 , \sigma'_2]$ is the boundary Dehn twist on $D$. Extend $\sigma'_1, \sigma'_2$ to diffeomorphisms of $X$ by having them act trivially outside of $D$. Suppose that the boundary Dehn twist on $X_0$ is smoothly isotopic to the identity relative $\partial X_0$. Then it follows that $\tau = [\sigma'_1 , \sigma'_2]$ regarded as a diffeomorphism of $X$ is smoothly isotopic to the identity by an isotopy which is the identity on a neighbourhood of $x$ (since we can slide $\tau$ over to $X_0$ and identify it with the boundary Dehn twist for $X_0$). A choice of such an isotopy determines a smooth family $\pi : E \to B = T^2$. Since the diffeomorphisms $\sigma'_1 , \sigma'_2$ and the chosen isotopy of $[\sigma'_1 , \sigma'_2]$ fix the point $x$, the family comes with a natural choice of section $s : B \to E$ whose normal bundle $N \to B$ is the flat bundle corresponding to the representation $\pi_1(T^2) \to SO(4)$ where the two generators of $\pi_1(T^2) \cong \mathbb{Z}^2$ are mapped to $\sigma_1, \sigma_2$. In particular, this implies that $w_2(N) \neq 0$ by a computation similar to that in the proof of \cref{pro found diffeo}.

Let $\mathfrak{s}$ be a spin$^c$ structure on $X$ and set $c_X = c_1(\mathfrak{s}) \in H^2(X ; \mathbb{Z})$. Clearly the diffeomorphisms $\sigma'_1,\sigma'_2$ preserve the isomorphism class of $\mathfrak{s}$ (in fact $\sigma'_1,\sigma'_2$ act trivially on $H^2(X ; \mathbb{Z})$). As in the proof of Theorem \ref{general thm}, there exists a spin$^c$ structure $\mathfrak{s}_E$ on the vertical tangent bundle of $E$ which restricts to $\mathfrak{s}$ on the fibers. Set $c = c_1(\mathfrak{s}_E) \in H^2(E ; \mathbb{Z})$. Let $\slashed{D}_E  \in K^0(B)$ denote the families index of the family of spin$^c$-Dirac operators associated to $(E , \mathfrak{s}_E)$.

\begin{lem}\label{lem:c1d0}
If $c_X$ is divisible by $32$ and $\sigma(X) = 16$ {\rm mod 32} then $c_1(\slashed{D}_E ) \neq 0$ {\rm mod 2}.
\end{lem}
\begin{proof}
As in the proof of Proposition \ref{prop: divisibility}, since $c_X$ is divisible by $32$ we can write $c = 32c_0 + \pi^*(b)$ for some $c_0 \in H^2(E ; \mathbb{Z})$ and some $b \in H^2(B ; \mathbb{Z})$. Then
\begin{align*}
\int_E (c^3 - cp_1) &= \int_E \pi^*(b)^3 - \pi^*(b) p_1 \; {\rm mod \; 32} \\
&= - 3\sigma(X) \int_B b \; {\rm mod \; 32}.
\end{align*}
Then since $c_1(\slashed{D}_E ) = (1/48) \int_{E/B} (c^3 - cp_1)$ (see the proof of Proposition \ref{prop: divisibility}) and $\sigma(X) = 16$ mod 32, we get $c_1(\slashed{D}_E ) = b = s^*(c)$ mod 2. Now we observe that since $c = c_1(\fraks_E)$ and $s^*(TE/B) \cong N$, the mod $2$ reduction of $s^*(c)$ is the second Stiefel--Whitney class of the normal bundle of $s$, which as shown above is non-zero. Hence $c_1(\slashed{D}_E ) \neq 0$ mod 2.
\end{proof}

\begin{thm}\label{thm:dehntw}
Let $X$ be a compact, simply-connected, smooth spin $4$-manifold. Suppose that $\mathfrak{s}$ is a spin$^c$ structure on $X$ such that $c_1(\mathfrak{s})$ is divisible by $32$ and $\SW(X , \mathfrak{s})$ is odd. Assume also that $\sigma(X) = 16$ {\rm mod 32} and $b_+(X) = 3$ {\rm mod 4}. Let $X_0$ be $X$ with an open ball removed. Then the Dehn twist on the boundary of $X_0$ is non-trivial in $\pi_0(\Diff(X_0 , \partial))$. 
\end{thm}
\begin{proof}
Suppose that the Dehn twist is trivial. Then as above we get a family $\pi : E \to B$ such that $c_1(\slashed{D}_E ) \neq 0$ mod 2 by Lemma \ref{lem:c1d0}. On the other hand the monodromy action on $H^2(X ; \mathbb{Z})$ is trivial, so $w_2(H^+(E)) = 0$. But since $b_+(X) = 3$ mod 4 and $\SW(X , \mathfrak{s}_X)$ is odd, we must have $c_1(\slashed{D}_E ) = w_2(H^+(E))$ mod 2 by Theorem \ref{thm: BK input}, a contradiction.
\end{proof}

\section{Examples}
\label{section Examples}

Theorems \ref{main thm intro} and \ref{main thm intro 3} stated in the introduction are a collection of the results proven in this section.

\subsection{Elliptic surfaces}

\label{subsection Elliptic}

We fix the notation for elliptic surfaces and recall basics following \cite[Subsection 3.3]{GS99}.
Let $E(n)$ denote the simply-connected elliptic surface of degree $n\geq1$ without multiple fibers, and $E(n)_{i,j}$ denotes the logarithmic transformation of $E(n)$ with multiplicites $(i,j)$, where $i,j\geq1$ are coprime.
Note that $E(n)_{i,1}=E(n)_i$, $E(n)_1=E(n)$.
Basic topological invariants of $E(n)_{i,j}$ are given by $\sigma(E(n)_{i,j}) = -8n$ and $b_+(E(n)_{i,j})=2n-1$.
In particular, if $n$ is even, $b_+(E(n)_{i,j}) = 3$ mod 4.
The elliptic surface $E(n)_{i,j}$ is always simply-connected, and this is spin if and only if $n$ is even and $ij$ is odd.
If $n>1$, then $E(n)_{i,j}$ is always minimal.

The Seiberg--Witten invariants of $X = E(n)_{i,j}$ may be described as follows (see, for example, \cite[Lecture 2]{FS6lec} or \cite[Chapter 3]{Ni20}). There exists a primitive class $t \in H^2( X ; \mathbb{Z} )$ such that the spin$^c$ structures $\mathfrak{s}$ with non-zero Seiberg--Witten invariant are precisely those with 
\begin{equation}\label{equ:ellipticbasic}
c_1(\mathfrak{s}) = (nij - 2ijk - 2ja - 2ib - i - j)t,
\end{equation}
where $0 \le k \le n-2$, $0 \le a \le i-1$, $0 \le b \le j-1$. Furthermore if $c_1(\mathfrak{s})$ is as above then
\begin{equation}\label{equ:ellipticsw}
\SW(X , \mathfrak{s}) = (-1)^k \binom{n-2}{k}.
\end{equation}

\begin{lem}\label{lem:ij}
Let $j \ge i \ge 1$ be odd coprime integers. Define
\[
S_1 = \{ (1,j) \}_{j \le 9} \cup \{ (3,5) \}, \; S_2 = \{ (1,j) \}_{j \le 15} \cup \{ (3,5) , (3,7) , (7,9), (5,11), (3,13)\}.
\]
\begin{itemize}
\item[(1)]{Suppose that $(i,j) \notin S_1$. Then for each $c \in \mathbb{Z}$, there exists integers $a,b$, $0 \le a \le i-1$, $0 \le b \le j-1$ and $k \in \{0,1\}$ such that $ja+ib+2ijk = c$ {\rm mod 16}.}
\item[(2)]{Suppose that $(i,j) \notin S_2$. Then for each $c \in \mathbb{Z}$, there exists integers $a,b$, $0 \le a \le i-1$ and $0 \le b \le j-1$ such that $ja+ib = c$ {\rm mod 16}.}
\end{itemize}
\end{lem}
\begin{proof}
We give the proof for (1), the case (2) is similar. Since $i,j$ are odd, they are invertible mod 16. If $j > 15$, then for any given $c$ we can take $a=k=0$ and take $b$ to be the unique integer in $\{0,1,\dots , 15\}$ such that $ib = c$ mod 16. So it suffices to consider only the cases where $j \le 15$. There are finitely many such cases.

If $(i,j) = (1,15)$, then letting $a=0$, letting $b$ run from $0$ and $k$ run from $0$ to $1$, we get $\{ b \}_{0 \le b \le 15} \cup \{ b-2\}_{0 \le b \le 15}$ which clearly covers every residue class mod $16$. A similar argument works for $(i,j) = (1,13)$ and $(1,11)$. The remaining cases $(i,j) = (5,9), (7,9), (3,7), (5,7)$ can be checked directly.
\end{proof}

\begin{thm}\label{thm:E4mij}
Let $X$ be $E(4m)_{i,j}$ where $m,i,j \ge 1$, $i \le j$ are coprime, $ij$ is odd and $(i,j)$ does not belong to the set $S_1$ in Lemma \ref{lem:ij}. Then $\pi_0(\Diff(X)) \to \Gamma(X)$ does not split.  
\end{thm}

\begin{proof}
The signature of $X$ is divisible by $32$ and $b_+(X) = 3$ mod 4, $b_+(X) > 3$. If we can find a spin$^c$ structure $\mathfrak{s}$ with $c_1(\mathfrak{s})$ divisible by $32$ and $\SW(X , \mathfrak{s})$ odd, then the result follows from \cref{general thm} and \cref{pro found diffeo}. According to Equation (\ref{equ:ellipticbasic}) we can assume $c_1(\mathfrak{s})$ is divisible by $(4mij - 2ijk - 2ja - 2ib - i - j)$, where $0 \le k \le 4m-2$, $0 \le a \le i-1$, $0 \le b \le j-1$ and $\SW(X , \mathfrak{s}) = (-1)^k \binom{4m-2}{k}$ by Equation (\ref{equ:ellipticsw}). We will assume that $k=0$ or $2$ as this guarantees that $\SW(X , \mathfrak{s})$ is odd. So $k=2k_0$ where $k_0 \in \{0,1\}$. Then it remains to choose $a,b$ such that $(4mij - 4ijk_0 - 2ja - 2ib - i - j)$ is divisible by $32$. Equivalently, we must solve
\[
ja + ib + 2ijk_0 = c \; {\rm mod \; 16}
\]
where $c = 2mij + (i+j)/2$ subject to $0 \le a \le i-1$, $0 \le b \le j-1$, $k_0 \in \{0,1\}$. Lemma \ref{lem:ij} (1) says this is possible since $(i,j) \notin S_1$.
\end{proof}

\begin{thm}
Let $X$ be $E(4m-2)_{i,j}$ where $m,i,j \ge 1$, $i \le j$ are coprime, $ij$ is odd and $(i,j)$ does not belong to the set $S_2$ in Lemma \ref{lem:ij}. Let $X_0$ be $X$ with an open ball removed. Then the boundary Dehn twist on $X_0$ is not smoothly isotopic to the identity relative boundary.
\end{thm}
\begin{proof}
We have $\sigma(X) = -32m+16 = 16$ mod 32 and $b_+(X) = 3$ mod 4. If we can find a spin$^c$ structure $\mathfrak{s}$ with $c_1(\mathfrak{s})$ divisible by $32$ and $\SW(X , \mathfrak{s})$ odd, then the result follows from \cref{thm:dehntw}. Similar to the proof of \cref{thm:E4mij} we can assume $c_1(\mathfrak{s})$ is divisible by $( (4m-2)ij - 2ijk - 2ja - 2ib -i -j)$ where $0 \le k \le 4m-4$, $0 \le a \le 1-i$, $0 \le b \le j-1$ and that $\SW(X , \mathfrak{s}) = (-1)^k \binom{4m-4}{k}$. This time we take $k=0$ to ensure that $\SW(X , \mathfrak{s})$ is odd. Then it remains to choose $a,b$ such that $( (4m-2)ij - 2ja - 2ib - i - j)$ is divisible by $32$ or equivalently, we must solve
\[
ja + ib = c \; {\rm mod \; 16}
\]
where $c = (2m-1)ij + (i+j)/2$ subject to $0 \le a \le i-1$, $0 \le b \le j-1$. Lemma \ref{lem:ij} (2) says this is possible since $(i,j) \notin S_2$.
\end{proof}

\subsection{Complete intersections}
\label{subsection Complete intersections}

Let $X \subseteq \mathbb{CP}^n$ be a complete intersection of multi-degree $(d_1, \dots , d_{n-2})$, $d_i \ge 2$. Then $X$ is simply-connected and is spin if and only if $\sum d_i - (n+1)$ is even. Moreover, we have (see, for example, \cite[Exercises 1.3.13]{GS99}):
\[
K_X = \mathcal{O}\left.\left( \sum_i d_i - (n+1) \right)\right|_X,
\]
\[
\sigma(X) = \frac{1}{3}\left( (n+1) - \sum_i d_i^2 \right) \prod_i d_i.
\]

Furthermore, if $\sum_i d_i - (n+1) \neq 0$ then the divisibility of $c_1(X)$ is exactly $\sum_i d_i - (n+1)$.

\begin{thm}
\label{thm complete intersection}
Let $X$ be a complete intersection $X$ such that $\sigma(X)$ is divisible by 32 and $\sum_i d_i = n+1$ {\rm mod 32}.
Then
$\pi_0(\Diff(X)) \to \Gamma(X)$does not split.  
\end{thm}

\begin{proof}
We will first show that $b_+(X) = 3$ mod 4. Since $b_+(X) = 1+2p_g(X)$, it is equivalent to show that $p_g(X)$ is odd. By assumption $c_1(X)$ and $\sigma(X)$ are divisible by $32$. Since $c_1(X)^2 = 12\chi(X) - e(X)$ and $\sigma(X) = 4\chi(X) - e(X)$, we have that $e(X) = 4\chi(X) = 12\chi(X)$ mod 32, hence $8\chi(X) = 0$ mod 32. So $\chi(X) = 1+p_g(X)$ is even and $p_g(X)$ is odd. In fact, our argument shows that $p_g(X) = 3$ mod 4 and hence $b_+(X) = 7$ mod 8. In particular, $b_+(X) > 3$.

Since $b_+(X) > 3$, $X$ cannot be a rational surface or a $K3$ surface. Then, by the classification of complete intersections (see such as \cite[Theorem 3.4.24]{GS99}), $X$ is a minimal surface of general type. The canonical spin$^c$ structure $\fraks_0$ has odd Seiberg--Witten invariant \cite[Theorem 7.4.1]{Mo96} and our assumption that $\sum_i d_i = n+1$ mod 32 ensures that $c_1(\fraks_0)$ is divisible by 32. Thus the assertion follows from \cref{general thm} and \cref{pro found diffeo}.
\end{proof}

\begin{thm}\label{thm complete intersection 2}
Let $X$ be a complete intersection $X$ such that $\sigma(X) = 16$ {\rm mod 32} and $\sum_i d_i = n+1$ {\rm mod 32}. Let $X_0$ be $X$ with an open ball removed. Then the boundary Dehn twist on $X_0$ is not smoothly isotopic to the identity relative boundary.
\end{thm}
\begin{proof}
Using the same argument as in the proof of \cref{thm complete intersection}, we get that $b_+(X) = 3$ mod 4 (in fact, $b_+(X) = 3$ mod 8), which rules out the possibility of $X$ being a rational surface. The canonical spin$^c$ structure $\fraks_0$ has odd Seiberg--Witten invariant and our assumption that $\sum_i d_i = n+1$ mod 32 ensures that $c_1(\fraks_0)$ is divisible by 32. Thus the assertion follows from \cref{thm:dehntw}.
\end{proof}

It is easy to see that many examples of complete intersections satisfy the assumptions of \cref{thm complete intersection} or \cref{thm complete intersection 2}. For simplicity, we consider complete intersections cut out by one or two polynomials. For a hypersurface $S_d$ in $\CP^3$ of degree $d\geq1$, we have:
\[
\sigma(S_d) = -\frac{1}{3}d(d^2-4).
\]
 
\begin{cor}
Let $X$ be a hypersurface $S_{d} \subset \CP^3$ of degree $d$. If $d = 4$ {\rm mod 32} then the boundary Dehn twist on $X_0$ is non-trivial.
\end{cor}

\begin{proof}
The condition $d = 4$ mod 32 implies that $c_1(X)$ is divisible by $32$. By \cref{thm complete intersection 2}, it suffices to check that $\sigma(X) = 16$ mod 32. We have
\begin{align*}
3\sigma(X) &= -d(d^2 - 4) \\
&= -4( 16 - 4) \; {\rm mod \; 32} \\
&= 16 \; {\rm mod \; 32}
\end{align*}
which shows that $\sigma(X) = 16$ mod 32.
\end{proof}

Lastly, we consider a complete intersection $S_{d_1,d_2} \subset \CP^4$ of multi-degree $d_1,d_2$. Then
\[
\sigma(S_{d_1,d_2}) = -\frac{1}{3} d_1 d_2 ( d_1^2 + d_2^2 - 5).
\]

\begin{cor}
Let $X = S_{d_1,d_2}$ where $d_1 + d_2 = 5$ {\rm mod 32}, and $d_1$ is divisible by $8$. Then $\pi_0(\Diff(X)) \to \Gamma(X)$ does not split.
\end{cor}
\begin{proof}
Since $d_1 + d_2 = 5$ mod 32, we have that $c_1(X)$ is divisible by $32$. Hence the result will follow form \cref{thm complete intersection} if we can show that $\sigma(X)$ is divisible by $32$. Since $-3\sigma(X) = d_1 d_2 (d_1^2 + d_2^2 - 5)$ and since $8 | d_1$, it suffices to show that $(d_1^2 + d_2^2 - 5)$ is divisible by $4$. Since $d_1$ is even and $d_1 + d_2 = 5$ mod 32, we must have that $d_2$ is odd and then $(d_1^2 + d_2^2 - 5) = (0 + 1 - 5) = 0$ mod 4, which gives the result.
\end{proof}

\bibliographystyle{plain}
\bibliography{mainref}
\end{document}